\documentclass{amsart}
\usepackage{amssymb}
\usepackage{amsfonts}

\newtheorem{theorem}{Theorem}
\theoremstyle{plain}

\newtheorem{lemma}{Lemma}

\newtheorem{proposition}{Proposition}

\numberwithin{equation}{section}

\begin{document}
\title{Design Lines}
\author{Harold N. Ward}
\address{Department of Mathematics\\
University of Virginia\\
Charlottesville, VA 22904\\
USA}
\email{hnw@virginia.edu}

\begin{abstract}
The two basic equations satisfied by the parameters of a block design define
a three-dimensional affine variety $\mathcal{D}$ in $\mathbb{R}^{5}$. A
point of $\mathcal{D}$ that is not in some sense trivial lies on four lines
lying in $\mathcal{D}$. These lines provide a degree of organization for
certain general classes of designs, and the paper is devoted to exploring
properties of the lines. Several examples of families of designs that seem
naturally to follow the lines are presented.
\end{abstract}

\maketitle

\section{Introduction\label{sectintro}}

Let $\mathcal{D}$ be the solution set in $\mathbb{R}^{5}$ of the two
equations connecting the parameters $v,b,r,k,\lambda $ of a balanced
incomplete block design:%
\begin{equation}
\mathcal{D=}\left\{ (v,b,r,k,\lambda ):vr=bk,\;r(k-1)=\lambda (v-1)\right\} .
\label{deseq}
\end{equation}%
We shall be interested in lines in $\mathcal{D}$, \textbf{design lines},
motivated in part by a consideration of \emph{Metis designs} satisfying the
relation $v=r+k+1$. Metis designs arose in \cite{MMW} on the way to a
variant of the construction of certain quasi-symmetric designs found by
Bracken, McGuire, and Ward \cite{BMW}. Symmetric Hadamard designs are among
designs with $v=r+k+1$, for which $v=b=4\lambda +3$ and $r=k=2\lambda +1$.
In most of this paper we shall consider only the parameters and not ask that
a point of $\mathcal{D}$ correspond to a genuine design--one that actually
has a model; if it does, we shall say the design \textbf{exists}. However,
we shall use the symbol $\Delta $ to represent a point of $\mathcal{D}$, and
we shall still refer to such a point $\Delta =(v,b,r,k,\lambda )$ on $%
\mathcal{D}$ as a design (if $\Delta $ is symmetric, we may use the
traditional shorthand $\Delta =(v,k,\lambda )$). It is not at all clear
whether the stratification of designs produced by the lines is of any real
significance. Nevertheless, a number of families of designs seem to respect
the lines. In Section \ref{sectfam} we give examples illustrating the
implied taxonomy.

The \textbf{complement} $\Delta C$ of a design $\Delta =(v,b,r,k,\lambda )$
is the design obtained (for a genuine design) by replacing each block by its
complement. Its parameters are given by 
\begin{equation*}
\Delta C=(v,b,b-r,v-k,b-2r+\lambda ),
\end{equation*}%
and if $\Delta \in \mathcal{D}$, then $\Delta C\in \mathcal{D}$ also. Some
other designs created from $\Delta $ when it is genuine are \textbf{multiple 
}designs, obtained by allowing each block to appear $m$ times. This design $%
m\times \Delta $ has parameters $(v,mb,mr,k,m\lambda )$; and for any $\Delta
\in \mathcal{D}$ and arbitrary $m\in \mathbb{R}$, we define $\Delta
M_{m}=(v,mb,mr,k,m\lambda )$, which is also in $\mathcal{D}$. Provided that $%
b,r$, and $\lambda $ are not all 0, the set $\{\Delta M_{m}:m\in \mathbb{R\}}
$ is a line in $\mathcal{D}$, the \textbf{multiple line} through $\Delta $.
It is the first of our design lines.

\section{Planes in $\mathcal{D}$}

In this section we shall show that there are seven planes in $\mathcal{D}$,
the $\mathcal{D}$\textbf{-planes}, whose members may be considered as
degenerate. For $\Delta =(v,b,r,k,\lambda )$ we define $Q=Q(\Delta
)=r^{2}-\lambda b$. The design equations give what we shall call the $Q$%
\textbf{-equations}:%
\begin{eqnarray}
Qv &=&b(r-\lambda )  \label{qeqns} \\
Qk &=&r(r-\lambda )  \notag \\
Q(v-1) &=&r(b-r)  \notag \\
Q(k-1) &=&\lambda (b-r)  \notag
\end{eqnarray}%
The difference $r-\lambda $ is the \textbf{order} of $\Delta $, and $b-r$ is
the replication parameter of the complementary design. Notice that in a
genuine design, $Q\geq 0$; if $Q=0$, the design should be regarded as
degenerate.

\begin{proposition}
\label{q=0}The set $\mathcal{Q}=\left\{ \Delta \in \mathcal{D}|Q(\Delta
)=0\right\} $ is the union of four planes, one of which is the plane on the
three singular points of $\mathcal{D}$.
\end{proposition}

\begin{proof}
Those singular points are $\Gamma _{0}=(0,0,0,0,0)$, $\Gamma
_{1}=(1,0,0,0,0) $, and $\Gamma _{2}=(1,0,0,1,0)$. Suppose that $Q(\Delta
)=0 $. If $r=0$, then either $\lambda =0$ or $b=0$, and by the design
equations (\ref{deseq}), either $b=0$ or $k=0$. If $\lambda \neq 0$, then $%
v=1$. If $r\neq 0$, then $\lambda \neq 0$ and $b\neq 0$, too; and the $Q$%
-equations (\ref{qeqns}) imply that $b=r=\lambda $ and then $v=k$. Putting
all this together, we get that $\mathcal{Q}$ is the union of these four
planes:%
\begin{eqnarray}
\Pi _{0} &:&b=0,\;r=0,\;\lambda =0  \label{q0planes} \\
\Pi _{1} &:&b=0,\;r=0,\;v=1  \notag \\
\Pi _{2} &:&r=0,\;\lambda =0,\;k=0  \notag \\
\Pi _{3} &:&b=r=\lambda ,\;v=k.  \notag
\end{eqnarray}%
Plane $\Pi _{0}$ is the one on the three singular points.
\end{proof}

For further planes and the design lines, consider the projection $\pi
:(v,b,r,k,\lambda )\rightarrow (b,r,\lambda )$. Then $\pi $ is one-to-one on
the subset $\mathcal{D}-\mathcal{Q}$. For if $\pi (\Delta )=\pi (\Delta
^{\prime })$ with $\Delta \neq \Delta ^{\prime }$, it can only be that $%
Q(\Delta )=Q(\Delta ^{\prime })=0$, since with $Q\neq 0$, the equations (\ref%
{qeqns}) can be solved for $v$ and $k$ in terms of $b,r,\lambda $.

\begin{lemma}
\label{onetoone}Suppose that $\Phi $ is an affine flat in $\mathcal{D}$
other than a point. Then if $Q$ is not identically 0 on $\Phi $, $\pi $ is
one-to-one on $\Phi $.
\end{lemma}

\begin{proof}
If $\pi $ is not one-to-one on $\Phi $, then $Q$ must be identically 0 on
each fibre of $\pi $ on $\Phi $. But as $\Phi $ is the union of the fibres, $%
Q$ is identically 0 on $\Phi $.
\end{proof}

\noindent Now let $\Pi $ be a plane in $\mathcal{D}$ that is not in $%
\mathcal{Q}$. Since $\pi $ is one-to-one on $\Pi $, two of $b,r,\lambda $
are independent on $\Pi $ and the third is linear in them. It follows that $%
Q $ is quadratic on $\Pi $: if $\lambda =\alpha b+\beta r+\gamma $, then $%
Q=r^{2}-\alpha b^{2}-\beta rb-\gamma b$, while if $b=\alpha r+\beta \lambda
+\gamma $, then $Q=r^{2}-\alpha r\lambda -\beta \lambda ^{2}-\gamma \lambda $%
; and both of these expressions for $Q$ are quadratic. If $r=\alpha b+\beta
\lambda +\gamma $, then 
\begin{equation*}
Q=\alpha ^{2}b^{2}+\beta ^{2}\lambda ^{2}+(2\alpha \beta -1)b\lambda
+2\alpha \gamma b+2\beta \gamma \lambda +\gamma ^{2}.
\end{equation*}%
This fails to be quadratic only when $\alpha =\beta =2\alpha \beta -1=0$,
conditions that are inconsistent. The $Q$-equations (\ref{qeqns}) imply that 
$v$ and $k$ must be constant on $\Pi $, since the right sides have degree at
most 2. If $v\neq 0,1$, then $r=bk/v$ and $\lambda =bk(k-1)/(v(v-1))$,
contrary to the independence of two of $b,r,\lambda $. Similarly, if $k\neq
0,1$, then $r=\lambda (v-1)/(k-1)$ and $b=\lambda v(v-1)/(k(k-1))$, again
contrary to independence.

Running through the four possible combinations for the $v$ and $k$ values on 
$\Pi $, we obtain planes from three of them. However, if $v=0$ and $k=1$,
then $b=0$ and $\lambda =0$, and this defines the replicate line on $%
(0,0,1,1,0)$. The planes are given as follows:

\begin{proposition}
\label{laterplanes}There are three more $\mathcal{D}$-planes in addition to
those of Proposition \ref{q0planes}. They are:%
\begin{eqnarray}
\Pi _{4} &:&v=0,\;k=0,\ r=\lambda  \label{higherplanes} \\
\Pi _{5} &:&v=1,\ k=1,\ b=r  \notag \\
\Pi _{6} &:&v=1,\ k=0,\ r=0.  \notag
\end{eqnarray}
\end{proposition}

\noindent Notice that the third equation in the plane definitions is a
consequence of the first two. We shall call the designs on these seven $%
\mathcal{D}$-planes the \textbf{flat} designs. All others (including some
that should clearly be regarded as degenerate) are the \textbf{bumpy}
designs. Lines in $\mathcal{D}$ other than those in the seven planes will be
called $\mathcal{D}$-\textbf{lines}.

\section{Affine automorphisms\label{sectaut}}

Now we seek the group $\mathcal{A}(\mathcal{D})$ of affine transformations
of $\mathbb{R}^{5}$ that leave $\mathcal{D}$ invariant and thus permute
design lines. We write the transformations in matrix terms, so that they
have the form $\Delta \rightarrow \Delta A+\Gamma $, where $A$ is an
invertible 5 by 5 matrix and $\Gamma $ is a design (the image of $%
(0,0,0,0,0) $). We have some automorphisms in hand: complementation $C$,
with matrix%
\begin{equation*}
\left[ 
\begin{array}{ccccc}
1 & 0 & 0 & 1 & 0 \\ 
0 & 1 & 1 & 0 & 1 \\ 
0 & 0 & -1 & 0 & -2 \\ 
0 & 0 & 0 & -1 & 0 \\ 
0 & 0 & 0 & 0 & 1%
\end{array}%
\right] .
\end{equation*}%
And we have the multiple transformation $M_{m}$, with matrix%
\begin{equation*}
\left[ 
\begin{array}{ccccc}
1 & 0 & 0 & 0 & 0 \\ 
0 & m & 0 & 0 & 0 \\ 
0 & 0 & m & 0 & 0 \\ 
0 & 0 & 0 & 1 & 0 \\ 
0 & 0 & 0 & 0 & m%
\end{array}%
\right] .
\end{equation*}%
Both of these transformations have $\Gamma =0$. The $M_{m}$ make up the 
\textbf{multiple subgroup} $\mathcal{M}(\mathcal{D})$ of $\mathcal{A}(%
\mathcal{D})$, which is centralized by $C$.

In what follows, it will help to describe the intersections of the $\mathcal{%
D}$-planes. Here is a table (completed by symmetry). For a line, we list a
typical design on it in terms of a variable component and condense the
design quintuple.%
\begin{equation}
\begin{tabular}{ccccccc}
\multicolumn{7}{c}{$\text{Table I. Plane intersections}$} \\ \hline\hline
& $\Pi _{2}$ & $\Pi _{3}$ & $\Pi _{4}$ & $\Pi _{5}$ & $\Pi _{6}$ & $\Pi _{0}$
\\ \cline{2-7}
$\Pi _{1}$ & \multicolumn{1}{|c}{$10000$} & $10010$ & $\varnothing $ & $%
1001\lambda $ & $1000\lambda $ & $100k0$ \\ 
$\Pi _{2}$ & \multicolumn{1}{|c}{} & $00000$ & $0b000$ & $\varnothing $ & $%
1b000$ & $v0000$ \\ 
$\Pi _{3}$ & \multicolumn{1}{|c}{} &  & $0bb0b$ & $1bb1b$ & $\varnothing $ & 
$v00v0$ \\ 
$\Pi _{4}$ & \multicolumn{1}{|c}{} &  &  & $\varnothing $ & $\varnothing $ & 
$00000$ \\ 
$\Pi _{5}$ & \multicolumn{1}{|c}{} &  &  &  & $\varnothing $ & $10010$ \\ 
$\Pi _{6}$ & \multicolumn{1}{|c}{} &  &  &  &  & $10000$%
\end{tabular}
\notag
\end{equation}%
Let $T:\Delta \rightarrow \Delta A+\Gamma $ be a member of $\mathcal{A}(%
\mathcal{D})$; $T$ must permute the $\mathcal{D}$-planes. As $T$ also
permutes the singular points of $\mathcal{D}$, $T$ fixes $\Pi _{0}$. With $%
S=\left\{ \Gamma _{0},\Gamma _{1},\Gamma _{2}\right\} $, we have a
homomorphism $\mathcal{A}(\mathcal{D})\rightarrow \mathrm{Sym}(S)$.
Moreover, $\Gamma =\Gamma _{0}T\in S$. The $M_{m}$ fix each the members of $%
S $ and that characterizes them:

\begin{proposition}
The kernel of the homomorphism $\mathcal{A}(\mathcal{D})\rightarrow \mathrm{%
Sym}(S)$ is $\mathcal{M}(\mathcal{D})$.
\end{proposition}

\begin{proof}
Let $M:\Delta \rightarrow \Delta A+\Gamma $ be in the kernel. Then $\Gamma
=\Gamma _{0}M=\Gamma _{0}$. As the affine transformation $M$ fixes each
member of the spanning set $S$ of $\Pi _{0}$, $M$ fixes $\Pi _{0}$
pointwise. Since the other planes are determined by their intersections with 
$\Pi _{0}$, $M$ fixes all the $\mathcal{D}$-planes. Then appropriate
intersections lead to further images for $M$: 
\begin{equation*}
\begin{tabular}{cc}
intersection & effect of $M$ \\ 
$\Pi _{2}\cap \Pi _{4}$ & $(0,1,0,0,0)\rightarrow (0,x,0,0,0)$ \\ 
$\Pi _{3}\cap \Pi _{5}$ & $(1,1,1,1,1)\rightarrow (1,y,y,1,y)$ \\ 
$\Pi _{1}\cap \Pi _{5}$ & $(1,0,0,1,1)\rightarrow (1,0,0,1,z)$%
\end{tabular}%
,
\end{equation*}%
for some $x,y,z$. Working out combinations from the images now in hand gives%
\begin{equation*}
M=\left[ 
\begin{array}{ccccc}
1 & 0 & 0 & 0 & 0 \\ 
0 & x & 0 & 0 & 0 \\ 
0 & y-x & y & 0 & y-z \\ 
0 & 0 & 0 & 1 & 0 \\ 
0 & 0 & 0 & 0 & z%
\end{array}%
\right] .
\end{equation*}%
Then, for instance, the image of the complete design $(3,3,2,2,1)$ is $%
(3,x+2y,2y,2,2y-z)$. For the design equations on this image we get $%
0=vr-bk=2(y-x)$ and $0=\lambda (v-1)-r(k-1)=2(y-z)$; so $x=y=z$ and $M=$ $%
M_{x}$.
\end{proof}

Already having the transposition $(\Gamma _{0},\Gamma _{1})$ produced by $C$%
, we obtain the following:

\begin{proposition}
The image of $\mathcal{A}(\mathcal{D})$ in $\mathrm{Sym}(S)$ is all of $%
\mathrm{Sym}(S)$. In addition, $\mathcal{A}(\mathcal{D})=\mathcal{M}(%
\mathcal{D})\times \left\langle C,N\right\rangle $, with $N$ as below.
\end{proposition}

\begin{proof}
All that is needed is another transposition. The transformation%
\begin{equation*}
N:(v,b,r,k,\lambda )\rightarrow (1-k,\lambda ,r,1-v,b)
\end{equation*}%
is visibly in $\mathcal{A}(\mathcal{D})$ since it switches the two equations
in (\ref{deseq}) defining $\mathcal{D}$. Like $C$, $N$ centralizes $\mathcal{%
M}(\mathcal{D})$. One has $\Gamma _{0}N=\Gamma _{2}$ and $N$ produces the
transposition $(\Gamma _{0},\Gamma _{2})$. In matrix terms, $N$ is the
transformation 
\begin{equation*}
\Delta \rightarrow \Delta \left[ 
\begin{array}{ccccc}
0 & 0 & 0 & -1 & 0 \\ 
0 & 0 & 0 & 0 & 1 \\ 
0 & 0 & 1 & 0 & 0 \\ 
-1 & 0 & 0 & 0 & 0 \\ 
0 & 1 & 0 & 0 & 0%
\end{array}%
\right] +\Gamma _{2}.
\end{equation*}%
Unfortunately, if $\Delta $ is a genuine design, $\Delta N$ is not!

That $\mathcal{A}(\mathcal{D})=\mathcal{M}(\mathcal{D})\left\langle
C,T\right\rangle $ now follows from the preceding proposition, and the
product's being direct is easy to verify. Of course, $\left\langle
C,N\right\rangle \simeq S_{3}$.
\end{proof}

\noindent The permutations of the $\mathcal{D}$-planes produced by $C$ is $%
(23)(56)$ and by $N$ it is $(12)(45)$. Moreover, if $T\in \left\langle
C,N\right\rangle $, $Q(\Delta T)=Q(\Delta )$.

\section{Lines in $\mathcal{D}$\label{sectlines}}

Each flat design is on an infinite number of lines in $\mathcal{D}$ lying in
whatever plane contains the design, and every design not in $\Pi _{0}$ lies
on a replicate line in $\mathcal{D}$. We shall refine the search for other
lines by the nature of the restriction of $Q$ to a line. Lines for which $Q$
is identically 0 are the lines in $\mathcal{Q}$, so that we may assume that $%
Q$ is not identically 0 on a line $\Lambda $ under scrutiny. By Lemma \ref%
{onetoone}, $\pi $ is one-to-one on $\Lambda $, and some one of $b,r,\lambda 
$ could be used as a variable for $\Lambda $.

\subsection{Quadratic $Q$}

\begin{proposition}
\label{qquadline}Suppose that $Q$ is quadratic, in a chosen variable, on a
line $\Lambda $ in $\mathcal{D}$. Then $\Lambda $ is either a replicate line
or is in one of the $\mathcal{D}$-planes.
\end{proposition}

\begin{proof}
By the $Q$-equations (\ref{qeqns}), $v$ and $k$ must be constant on $\Lambda 
$ as with the analogous plane, since the degrees of the right sides are at
most 2. As in the comments after Lemma \ref{onetoone}, if $v\neq 0,1$, we
obtain $r=bk/v,\;\lambda =bk(k-1)/(v(v-1))$; and if $k\neq 0,1$, $b=\lambda
v(v-1)/(k(k-1)),\;r=\lambda (v-1)/(k-1)$. Both of these possibilities give
replicate lines. Thus we may take $v=0$ or $1$ and $k=0$ or $1$. All but $%
v=0,k=1$ produce a line in one of $\Pi _{4}$, $\Pi _{5}$, or $\Pi _{6}$.
When $v=0,\;k=1$, the line is the replicate line $v=b=0,\;k=1,\;\lambda =0$
mentioned earlier.
\end{proof}

\subsection{Constant $Q$}

Now let $\Lambda $ be a line in $\mathcal{D}$ on which $Q$ is a nonzero
constant. If $r$ is constant on $\Lambda $, then one of $b$ or $\lambda $
must be 0 (otherwise $Q$ would have positive degree). If $\lambda =0$ on $%
\Lambda $, then by $Q(k-1)=\lambda (b-r)$ in (\ref{qeqns}), $k=1$ and then $%
b=vr$. Since one of $r,b,\lambda $ can serve as a variable on $\Lambda $, it
must be that $r\neq 0$ and we can use $v$ as the variable. The line is%
\begin{equation*}
\Lambda _{0}:b=fv,\;r=f,\;k=1,\;\lambda =0
\end{equation*}%
with $f\neq 0$ (the choice of notation here and for the other constant $Q$
lines will be explained later). If $r$ is constant and $b=0$ on $\Lambda $,
then $\lambda $ can be the line variable. This time $Qv=b(r-\lambda )$
implies that $v=0$; $Q(v-1)=r(b-r)$ gives $r\neq 0$. The line is%
\begin{equation*}
\Lambda _{P}:v=0,\;b=0,\;r=-p,\;k=1+\lambda /p.
\end{equation*}%
Finally, suppose that $r$ is not constant on $\Lambda $, so that as $Q$\emph{%
\ is} constant, neither $b$ nor $\lambda $ is. Now the $Q$-equations imply
that both $r-\lambda $ and $b-r$ are constant. Put $b=\lambda +e$, $%
r=\lambda +f$, making $Q=(2f-e)\lambda +f^{2}$; the constancy of $Q$
requires $e=2f\neq 0$. Thus the line is 
\begin{equation*}
\Lambda _{1}:v=\lambda /f+2,\;b=\lambda +2f,\;r=\lambda +f,\;k=\lambda /f+1.
\end{equation*}%
These three types of lines are indeed in $\mathcal{D}$, and we collect them
here:

\begin{proposition}
\label{qconstline}There are three one-parameter families of lines in $%
\mathcal{D}$ on which $Q$ is a nonzero constant:%
\begin{eqnarray*}
\Lambda _{P}(p) &:&v=0,\;b=0,\;r=-p,\;k=1+\lambda /p;\;p\neq 0 \\
\Lambda _{0}(f) &:&b=fv,\;r=f,\;k=1,\;\lambda =0;\;f\neq 0 \\
\Lambda _{1}(f) &:&v=\lambda /f+2,\;b=\lambda +2f,\;r=\lambda +f,\;k=\lambda
/f+1;\;f\neq 0.
\end{eqnarray*}
\end{proposition}

\noindent The designs on the first two types of lines clearly rate as
degenerate; and those on a $\Lambda _{1}$-line are complements of designs on 
$\Lambda _{0}$-lines (with the same value of $f$).

\subsection{Linear $Q$}

We come to the most interesting lines, those on which $Q$ is linear (in a
chosen line variable). We analyze some preliminary possibilities. For
reference, here are the $Q$-equations (\ref{qeqns}) again:%
\begin{eqnarray*}
Qv &=&b(r-\lambda ) \\
Qk &=&r(r-\lambda ) \\
Q(v-1) &=&r(b-r) \\
Q(k-1) &=&\lambda (b-r).
\end{eqnarray*}%
Suppose first that $r$ is constant on the line $\Lambda $ in $D$. Then $%
Q=r^{2}-\lambda b$ implies that one of $\lambda $ or $b$ must be a nonzero
constant and the other linear. If $\lambda $ is constant, then $%
Qk=r(r-\lambda )$ requires $k=0$ and either $r=0$ or $r=\lambda $. If $r=0$,
then $r(k-1)=\lambda (v-1)$ gives $v=1$ and $\Lambda \subset \Pi _{6}$, from
(\ref{higherplanes}), with $Q=-\lambda b$. If $r=\lambda $, then $%
Qv=b(r-\lambda )$ makes $v=0$ and $\Lambda \subset \Pi _{4}$, with $Q=r(r-b)$%
. On the other hand, if $b$ is constant, then $Q(v-1)=r(b-r)$ implies that $%
v=1$ and either $r=0$ or $b=r$. If $r=0$, we are back at $k=0$, and $\Lambda
\subset \Pi _{6}$. If $b=r$, then $Q(k-1)=\lambda (b-r)$ gives $k=1$, and $%
\Lambda \subset \Pi _{5}$; $Q=r(r-\lambda )$.

Now suppose that $r$ is not constant on $\Lambda $, so that neither $\lambda 
$ nor $b$ is constant, but suppose also that one of $r-\lambda $ or $b-r$ is
constant. If $r-\lambda $ is constant, then $Qv=b(r-\lambda )$ and $%
Qk=r(r-\lambda )$ imply that $v$ and $k$ are constant, and then so is $b-r$.
Similarly, if $b-r$ is constant, then $v$, $k$, and $r-\lambda $ are
constant. Thus we are supposing that all of $r-\lambda ,b-r,v$, and $k$ are
constant. As above, put $b=\lambda +e$ and $r=\lambda +f$. Again, $%
Q=(2f-e)\lambda +f^{2}$, but now it must be that $2f-e\neq 0$. The design
equation $vr=bk$ reads $v\lambda +vf=k\lambda +ke$, which requires $v=k$ and
then $v(f-e)=0$. Similarly, $r(k-1)=\lambda (v-1)$ becomes $(v-1)f=0$. By $%
v(f-e)=0$ we have either $v=0$ or $f=e$. If $v=0$, then $f=0$, and the line
is in $\Pi _{4}$. If $v\neq 0$, then $f=e$; and $(v-1)f=0$ means either $v=1$
or $f=0$. If $v=1$, then $k=1$ and the line is in $\Pi _{5}$. But if $v\neq
1 $, then $f=e=0$. So $b=r=\lambda $, and that would actually make $Q=0$.

Therefore the new possibilities are that \emph{all} of $b,r,\lambda
,r-\lambda ,$ and $b-r$ are linear on $\Lambda $. The $Q$-equations then
imply that $v$ and $k$ are also linear. The factors on the right in the $Q$%
-equations must match those on the left up to constant multiples. There are
three possibilities for $Q$: a multiple of $b$ and $r$; a multiple of $r$
and $\lambda $; and a multiple of $r-\lambda $ and $b-r$. Each choice gives
a two-parameter family of lines, and we label the families and the
parameters as follows:

\begin{eqnarray}
P(f,p) &:&Q=(f-p)r=p(b-r)  \label{lineeqs} \\
F_{0}(f,p) &:&Q=(f-p)r=f(r-\lambda )  \notag \\
F_{1}(f,p) &:&Q=f(r-\lambda )=p(b-r)  \notag
\end{eqnarray}%
For the first family we have used $r$ and $b-r$ instead of $b$ and $r$; and
in the second, $r$ and $r-\lambda $. This was done for the evident
uniformity. Both $f$ and $p$ are nonzero, and $f\neq p$. The three line
families of Proposition \ref{qconstline} can be regarded as limiting cases
for the families just obtained: 
\begin{eqnarray}
\Lambda _{P}(p) &=&P(0,p)  \label{limitlines} \\
\Lambda _{0}(f) &=&F_{0}(f,0)  \notag \\
\Lambda _{1}(f) &=&F_{1}(f,f).  \notag
\end{eqnarray}%
To see that the $P$-, $F_{0}$-, and $F_{1}$-lines just obtained genuinely
lie in $\mathcal{D}$, we describe them using $Q$ itself as the line
variable, with $d=f-p$:%
\begin{equation}
\begin{tabular}{cccc}
\multicolumn{4}{c}{Table II. Line equations in $Q$} \\ \hline\hline
& $P(f,p)$ & $F_{0}(f,p)$ & $F_{1}(f,p)$ \\ \cline{2-4}
$v=$ & \multicolumn{1}{|c}{$\dfrac{Q}{pd}+1$} & $\dfrac{Q}{pd}-\dfrac{d}{p}$
& $\dfrac{Q}{pd}-\dfrac{p}{d}$ \\ 
$b=$ & \multicolumn{1}{|c}{$\dfrac{fQ}{pd}$} & $\dfrac{fQ}{pd}-\dfrac{fd}{p}$
& $\dfrac{fQ}{pd}-\dfrac{fp}{d}$ \\ 
$r=$ & \multicolumn{1}{|c}{$\dfrac{Q}{d}$} & $\dfrac{Q}{d}$ & $\dfrac{Q}{d}-%
\dfrac{fp}{d}$ \\ 
$k=$ & \multicolumn{1}{|c}{$\dfrac{Q}{fd}+\dfrac{p}{f}$} & $\dfrac{Q}{fd}$ & 
$\dfrac{Q}{fd}-\dfrac{p}{d}$ \\ 
$\lambda =$ & \multicolumn{1}{|c}{$\dfrac{pQ}{fd}-\dfrac{pd}{f}$} & $\dfrac{%
pQ}{fd}$ & $\dfrac{pQ}{fd}-\dfrac{fp}{d}$%
\end{tabular}
\label{qparam}
\end{equation}%
Notice that the three lines with the same $f$ and $p$ are parallel. The $P$%
-, $F_{0}$-, and $F_{1}$-lines are determined by their intersections with $%
\Pi _{1}$, $\Pi _{2}$, and $\Pi _{3}$, respectively. They do not meet $\Pi
_{0}$.

\begin{theorem}
Each bumpy design is on exactly four lines in $\mathcal{D}$.
\end{theorem}

\begin{proof}
One of the lines on such a design $\Delta $ is the replicate line, a genuine
line since $\Delta \notin \Pi _{0}$. Any other line on $\Delta $ must be a $%
\Lambda $-, $P$-, or $F$-line, since all other lines lie in $\mathcal{D}$%
-flats, and unlike the replicate line, these do not meet $\Pi _{0}$.

With $\Delta =(v,b,r,k,\lambda )$, we have $Q=r^{2}-\lambda b\neq 0$. We
claim that $r\neq \lambda $, $b\neq r$, and $r\neq 0$. If $r=\lambda $, then
the $Q$-equations (\ref{qeqns}) imply that $v=0,k=0$, and $\Delta \in \Pi
_{4}$. If $b=r$, then $v=1,k=1$, and $\Delta \in \Pi _{5}$. Finally, $r=0$
implies that $v=1,k=0$, and $\Delta \in \Pi _{6}$. Now we can solve for the
parameters $f$ and $p$ in the $P$-, $F_{0}$-, and $F_{1}$-lines so as to
make them contain $\Delta $. Here are the equations (\ref{lineeqs}) recast
for $f$ and $p$, with simplified expressions included, their denominators
assumed not 0:%
\begin{equation}
\begin{tabular}{cccc}
\multicolumn{4}{c}{Table III. Values of $f$ and $p$} \\ \hline\hline
& $P$ & $F_{0}$ & $F_{1}$ \\ \cline{2-4}
$f=$ & \multicolumn{1}{|c}{$\dfrac{Qb}{r(b-r)}\;\left( =\dfrac{b}{v-1}%
\right) $} & $\dfrac{Q}{r-\lambda }\;\left( =\dfrac{b}{v}\right) $ & $\dfrac{%
Q}{r-\lambda }\;\left( =\dfrac{b}{v}\right) $ \\ 
$p=$ & \multicolumn{1}{|c}{$\dfrac{Q}{b-r}\;\left( =\dfrac{r}{v-1}\right) $}
& $\dfrac{Q\lambda }{r(r-\lambda )}\;\left( =\dfrac{\lambda }{k}\right) $ & $%
\dfrac{Q}{b-r}\;\left( =\dfrac{r}{v-1}\right) $%
\end{tabular}
\label{Table III}
\end{equation}%
Each type of line has an exceptional case that leads to the corresponding $%
\Lambda $-line. Thus for a $P$-line, $f\neq p$ and $p\neq 0$, correctly; but
it could be that $f=0$, corresponding to $b=0$. In that case, $p=-r$ and the
substitute line is $\Lambda _{P}(-r)$.
\end{proof}

On a $P$-line, $v=(f/p)k$ and $b=(f/p)r$; in a sense, there is a point
replication appearing, and that is the reason for \textquotedblleft $P$%
\textquotedblright . On the $F_{0}$- and $F_{1}$-lines, $b=fv$ and $r=fk$.
The \textbf{Fisher factor} $f$ measures departure from symmetric designs,
where $f=1$; it motivates the \textquotedblleft $F$\textquotedblright\
designation.

Direct computation with the equations (\ref{qparam}) verifies the following
proposition.

\begin{proposition}
\label{lemmalineimages}Members of $\mathcal{A}(\mathcal{D})$ have these
effects on lines: both $C$ and $N$ take replicate lines to replicate lines.
The transformation $C$ takes $P$-lines to $P$-lines and switches $F_{0}$-
and $F_{1}$-lines. Similarly, $N$ takes $F_{1}$-lines to $F_{1}$-lines and
switches $P$- and $F_{0}$-lines. The two transformations (and their
products) have uniform effects on the parameters: if $L$ is any of the
symbols $P,F_{0},F_{1}$ and $L(f,p)C=L^{\prime }(f^{\prime },p^{\prime })$,
then $f^{\prime }=f$ and $p^{\prime }=f-p$. Likewise, if $L(f,p)N=L^{\prime
}(f^{\prime },p^{\prime })$, then $f^{\prime }=-p$ and $p^{\prime }=-f$. In
addition, $L(f,p)M_{m}=L(mf,mp)$.
\end{proposition}

Given $f\neq 0$, the lines $F_{0}(f,p)$ and $F_{1}(f,p)$, $p\neq f$
(including $\Lambda _{0}(f)$ and $\Lambda _{1}(f)$), are rulings of the
hyperboloid of one sheet $\mathcal{H}$ defined by $Q=f(r-\lambda )$ in the
affine 3-flat $\Pi :v=b/f,k=r/f$; here $\mathcal{H}=\mathcal{D}\cap \Pi $.
Two lines on $\mathcal{H}$ are not obtained this way: $\Pi \cap \Pi _{2}$
and $\Pi \cap \Pi _{3}$. Application of $N$ and $C$ shows that similar
statements hold for the pairings $P,F_{1}$ with $Q=p(b-r)$, and $P$ $,F_{0}$
with $Q=(f-p)r$.

\section{Relations along a design line\label{sectrel}}

Suppose we ask for conditions under which a linear relation%
\begin{equation}
Vv+Bb+Rr+Kk+L\lambda =A  \label{linrel}
\end{equation}%
is valid along a $\mathcal{D}$-line (one not in any of the seven planes). On
substituting (\ref{qparam}), we must have an identity in $Q$ (or $m$ for a
replicate line), and that simply requires that the $Q$ or $m$ coefficient be
0 and that the relation holds for one value of $Q$ or $m$. However, because
of the presence of flat designs, we can answer this somewhat strange
question: are there relations that hold along a line of a specified type
through a bumpy design $\Delta _{0}$, given simply that the relation holds
at $\Delta _{0}$? What makes this work is that certain relations hold for
flat designs on each type of line. Thus if the relation holds at the bumpy
design $\Delta _{0}$ on a line, it holds at two designs on the line and is
therefore valid along the whole line.

The replicate line through $\Delta _{0}$ is%
\begin{equation*}
v=v_{0},\;b=mb_{0},\;r=mr_{0},\;k=k_{0},\;\lambda =m\lambda _{0},
\end{equation*}%
and this contains the design $(v_{0},0,0,k_{0},0)\in \Pi _{0}$. The relation 
$Bb+Rr+L\lambda =0$ holds on $\Pi _{0}$; so if this equation is true at $%
\Delta _{0}$, it is valid for the entire replicate line on $\Delta _{0}$.

Now we need the flat designs on the $P$-, $F_{0}$-, and $F_{1}$-lines (we
shall not consider the $\Lambda $-lines). They are listed here:%
\begin{equation*}
\begin{tabular}{ccc}
\multicolumn{3}{c}{$\text{TABLE IV. Flat designs on lines}$} \\ \hline\hline
line & $Q=0$ & $Q\neq 0$ \\ \cline{2-3}
$P(f,p)$ & \multicolumn{1}{|c}{$(1,0,0,\dfrac{p}{f},-\dfrac{pd}{f})\in \Pi
_{1}$} & $(0,-f,-p,0,-p)\in \Pi _{4}$ \\ 
$F_{0}(f,p)$ & \multicolumn{1}{|c}{$(-\dfrac{d}{p},-\dfrac{fd}{p},0,0,0)\in
\Pi _{2}$} & $(1,f,f,1,p)\in \Pi _{5}$ \\ 
$F_{1}(f,p)$ & \multicolumn{1}{|c}{$(-\dfrac{p}{d},-\dfrac{fp}{d},-\dfrac{fp%
}{d},-\dfrac{p}{d},-\dfrac{fp}{d})\in \Pi _{3}$} & $(1,f,0,0,-p)\in \Pi _{6}$%
\end{tabular}%
\end{equation*}%
We next list some relations that hold for the flat designs and the
corresponding lines on which the flat designs appear. For example, the row
for flat type 1 in the table means that a design in $\Pi _{1}$ satisfies the
relation $Vv+Bb+Rr=V$, and this applies to a $P$-line. These relations come
from substituting the coordinates of a typical flat design of a specified
type and demanding that the relation hold identically.%
\begin{equation*}
\begin{tabular}{cccccccc}
\multicolumn{8}{c}{$\text{TABLE\ V. Relations on lines}$} \\ \hline\hline
line type & $V$ & $B$ & $R$ & $K$ & $L$ & $=A$ & flat type \\ \cline{2-7}
$P$ & \multicolumn{1}{|c}{$V$} & $B$ & $R$ & $0$ & $0$ & $V$ & 
\multicolumn{1}{|c}{1} \\ 
$F_{0}$ & \multicolumn{1}{|c}{$0$} & $0$ & $R$ & $K$ & $L$ & $0$ & 
\multicolumn{1}{|c}{2} \\ 
$F_{1}$ & \multicolumn{1}{|c}{$V$} & $B$ & $R$ & $-V$ & $-B-R$ & $0$ & 
\multicolumn{1}{|c}{3} \\ 
$P$ & \multicolumn{1}{|c}{$V$} & $0$ & $R$ & $K$ & $-R$ & $0$ & 
\multicolumn{1}{|c}{4} \\ 
$F_{0}$ & \multicolumn{1}{|c}{$V$} & $B$ & $-B$ & $K$ & $0$ & $V+K$ & 
\multicolumn{1}{|c}{5} \\ 
$F_{1}$ & \multicolumn{1}{|c}{$V$} & $0$ & $R$ & $K$ & $0$ & $V$ & 
\multicolumn{1}{|c}{6}%
\end{tabular}%
\end{equation*}%
We shall refer to relations by the line and flat types, such as $(P,1)$.

\section{Taxonomy\label{sectfam}}

Here we present various families of designs, with numerical examples. The
data have been taken from the extensive parameter table in \cite[I.1.35]{CD}%
, which contains comments and references for the designs it lists (in the
table, $k\leq v/2$); we shall cite it as CRC. We often use $r$ as the line
variable, since the table is ordered by its value; other variables may be
used to simplify expressions. One can find $f$ and $p$ from Table III.

\subsection{Hadamard and Metis designs\label{subhm}}

As mentioned in the introduction, Metis designs satisfy $v=r+k+1$. This is
an $(F_{1},6)$ relation, so that all designs on the $F_{1}$ line through a
Metis design are such designs. For such an $F_{1}$-line, $p=f/(f+1)$.
Symmetric Metis designs are Hadamard designs, and they all lie on the 
\textbf{Hadamard line}%
\begin{equation*}
F_{1}(1,1/2):v=b=4Q-1,r=k=2Q-1,\lambda =Q-1.
\end{equation*}%
The smallest nonsymmetric Metis design is $(13,26,8,4,2)$. This design can
be realized as $2\times (13,4,1)$, the component being the projective plane
of order 3. Such a replicated realization of a nonsymmetric Metis design is
quite common and is discussed in \cite{MMW}.

The $F_{1}$-line on $(13,26,8,4,2)$ is 
\begin{equation*}
F_{1}(2,2/3):v=\frac{3}{2}r+1,\;b=3r+2,\;k=\frac{r}{2},\;\lambda =\frac{r-2}{%
3}.
\end{equation*}%
Some other existing Metis designs on this line are%
\begin{equation*}
\begin{tabular}{llll}
$(22,44,14,7,4)$ & $(31,62,20,10,6)$ & $(40,80,26,13,8)$ & $%
(49,98,32,16,10). $%
\end{tabular}%
\end{equation*}%
The existence of $(58,116,38,19,12)$ on this line is listed as unknown in
CRC. All but the first of the three can be realized as $2\times \Delta $
from a symmetric design $\Delta $. The symmetric design $(22,7,2)$ needed
for the first does not exist, by the Bruck-Ryser-Chowla theorem \cite[%
Theorem 2.5.10]{IS}--technically, by the $v$ even case, due to M. P. Sch\"{u}%
tzenberger \cite{Sch}, that then requires $n=r-\lambda =5$ to be a square.

\subsection{Quasi-residual designs\label{subqr}}

A more classical family is that of \textbf{quasi-residual} designs \cite[%
Section II.6.4]{CD}, for which $r=k+\lambda $. This is both a $(P,4)$ and an 
$(F_{0},2)$ relation, with $V=0$. The equivalent condition $b=v+r-1$ is a $%
(P,1)$ and an $(F_{0},5)$ relation. The lines involved are $P(f,f-1)$ and $%
F_{0}(f,f-1)$, and in the three-dimensional subspace of $\mathbb{R}^{5}$
that is the intersection of the two hyperplanes $r-k-\lambda =0$ and $%
v-b+r=1 $, the lines are rulings of the hyperboloid of one sheet $%
r(r-1)=\lambda b$.

For example, the (block) residual of a $(16,6,2)$ biplane is a $%
(10,15,6,4,2) $ design whose $P$-line is%
\begin{equation*}
P(5/3,2/3):v=\frac{3r+2}{2},\;b=\frac{5}{2}r,\;k=\frac{3r+2}{5},\;\lambda =%
\frac{2r-2}{5}.
\end{equation*}%
This contains the designs $(25,40,16,10,6)$, $(40,65,26,16,10)$, and $%
(55,90,36,22,14)$. The first two exist, with models as residuals of
symmetric designs, but the third is listed as unknown in CRC. The $F_{0}$%
-line on $(10,15,6,4,2)$ is%
\begin{equation}
F_{0}(3/2,1/2):v=2r-2,\;b=3r-3,\;k=\frac{2}{3}r,\;\lambda =\frac{1}{3}r
\label{f2}
\end{equation}%
and it contains these additional designs:%
\begin{equation*}
\begin{tabular}{lll}
$(4,6,3,2,1)$ & $(16,24,9,6,3)$ & $(22,33,12,8,4)\;\nexists $ \\ 
$(28,42,15,10,5)\ntriangleleft $ & $(34,51,18,12,6)$ & $(40,60,21,14,7)?$ \\ 
$(46,69,24,16,8)\ntriangleleft $ & $(52,78,27,18,9)$ & $(58,87,30,20,10)?$
\\ 
$(64,96,33,22,11)?$ & $(70,105,36,24,12)?$ & $(76,114,39,26,13)?$%
\end{tabular}%
\end{equation*}%
The designs that do exist can be realized as residuals of symmetric designs,
and the status of those with question marks is at present unknown. The $%
\nexists $-marked $(22,33,12,8,4)$ has recently been shown not to exist \cite%
{BLT} (see also the exposition in \cite{R}). Designs with the parameters
marked with $\ntriangleleft $ do exist, but none are residuals of symmetric
designs, the needed designs being ruled out by the Bruck-Ryser-Chowla
theorem. This family of designs on $F_{0}(3/2,1/2)$ is singled out in \cite[%
Section 3]{BLT} as providing especially interesting candidates for existence
investigations, in part because of its connections with codes.

If $\Delta =(v,b,r,k,\lambda )$ is a quasi-residual design, the parent
symmetric design (which of course may not really exist) is $\widetilde{%
\Delta }=(b+1,b+1,r,r,\lambda )$. Applying the suggested transformation $%
\Delta \rightarrow \widetilde{\Delta }$ to the $P$- and $F_{0}$-lines
through the quasi-residual design $\Delta _{0}$, both of which comprise
quasi-residual designs, must give lines of symmetric designs, since $%
\widetilde{\Delta }$ is symmetric Thus they are the $F_{0}$- and $F_{1}$%
-lines through $\widetilde{\Delta _{0}}$. The $r$ and $k$ along an $F_{0}$-
line are proportional to $\lambda $, which is not the case for an $F_{1}$%
-line. So it must be that the $F_{0}$-line on $\Delta _{0}$ maps to the $%
F_{0}$-line on $\widetilde{\Delta _{0}}$, and the $P$-line to the $F_{1}$.
With $\Delta =(10,15,6,4,2)$ and $\widetilde{\Delta }=(16,6,2)$ as above, we
find that the $F_{0}$-line on $\widetilde{\Delta }$ is%
\begin{equation*}
F_{0}(1,1/3):v=b=3r-2,\;k=r,\;\lambda =\frac{r}{3};
\end{equation*}%
and the $F_{1}$-line is%
\begin{equation*}
F_{1}(1,2/5):v=b=\frac{5}{2}r+1,\;k=r,\;\lambda =\frac{2r-2}{5}.
\end{equation*}%
Other existing designs on $F_{0}(1,1/3)$ are%
\begin{equation*}
\begin{tabular}{llll}
$(7,3,1)$ & $(25,9,3)$ & $(70,24,8)$ & $(79,27,9).$%
\end{tabular}%
\end{equation*}%
Further designs on this line in the table range are ruled out by the
Bruck-Ryser-Chowla theorem, except for $(97,33,11)$, whose existence is
unknown. Two further designs on $F_{1}(1,2/5)$ are $(41,16,6)$ and $%
(66,26,10)$; the third in the CRC range, $(91,36,14)$, does not exist, again
by the Bruck-Ryser-Chowla theorem.

For a discussion of the implied embedding problems, see \cite[Chapter 5]{IS}%
. A quasi-residual design with $\lambda =1$ is an affine plane and so the
residual of a projective plane. When $\lambda =2$, the Hall-Connor theorem 
\cite{HC} (see also \cite[Theorem 8.2.20]{IS}) guarantees that for such a
quasi-residual design, there is a symmetric design having it as a residual.
Now the form of the equations for $P$- and $F_{0}$-lines, in which certain
parameters are proportional to $Q$, hints at a possible replication-type
construction along these lines. But there does not seem to be one. For
example, the line%
\begin{equation*}
P(3/2,1/2):v=2Q+1,\;b=3Q,\;r=Q,\;k=(2Q+1)/3,\;\lambda =(Q-1)/3,
\end{equation*}%
contains the genuine design $(9,12,4,3,1)$ and the design $(15,21,7,5,2)$.
But this second quasi-residual design does not exist: it would be the
residual of a $(22,7,2)$ symmetric design, by the Hall-Connor theorem, and
this design does not exist, from Subsection \ref{subhm}. And as we saw, the
line $F_{0}(3/2,1/2)$ contains genuine designs for $Q=1$ and 2 but not $Q=4$.

\subsection{Derived designs and a family constructed from three-class
association schemes\label{subder}}

If $\widetilde{\Delta }=(\widetilde{v},\widetilde{k},\widetilde{\lambda })$
is a genuine symmetric design, a \textbf{derived} design $\Delta $ has a
block $\mathbf{B}$ of $\widetilde{\Delta }$ for its point set, and its
blocks are the intersections of $\mathbf{B}$ with the other blocks of $%
\widetilde{\Delta }$. The parameters of $\Delta $ are 
\begin{equation*}
v=\widetilde{k},\text{ }b=\widetilde{v}-1,\text{ }r=\widetilde{k}-1,\text{ }%
k=\widetilde{\lambda },\text{ }\lambda =\widetilde{\lambda }-1.
\end{equation*}%
Thus $k=\lambda +1$, and a design for which this holds is called \textbf{%
quasi-derived}\emph{\ }\cite[Section 13.4]{IS}. The other evident relation $%
v=r+1$ is a consequence; it is a $(P,1)$ and an $(F_{1},6)$ relation.
Quasi-derived designs and quasi-residual designs are complements of one
another. If $\Delta $ with $k=\lambda +1$ runs over a $P$-line, for which
necessarily $p=1$, it is%
\begin{equation*}
P(f,1):v=f\lambda +f,\text{ }b=f^{2}\lambda +f^{2}-f,\text{ }r=f\lambda +f-1,%
\text{ }k=\lambda +1.
\end{equation*}%
The parameters of the corresponding $\widetilde{\Delta }$ (in terms of $%
\lambda $ on $P(f,1)$) are%
\begin{equation*}
\widetilde{v}=\widetilde{b}=f^{2}\lambda +f^{2}-f+1,\text{ }\widetilde{r}=%
\widetilde{k}=f\lambda +f,\text{ }\widetilde{\lambda }=\lambda +1,
\end{equation*}%
and these lie on the $F_{0}$-line 
\begin{equation*}
F_{0}(1,1/f):\widetilde{v}=\widetilde{b}=f^{2}\widetilde{\lambda }-f+1,\text{
}\widetilde{r}=\widetilde{k}=f\widetilde{\lambda }.
\end{equation*}%
On the other hand, if $\Delta $ runs over the $F_{1}$-line%
\begin{equation*}
F_{1}(f,1):v=f\lambda +f+1,\text{ }b=f^{2}\lambda +f^{2}+f,\text{ }%
r=f\lambda +f,\text{ }k=\lambda +1,
\end{equation*}%
then the $\widetilde{\Delta }$ lie on the $F_{1}$-line%
\begin{equation*}
F_{1}(1,1/f):\widetilde{v}=\widetilde{b}=f^{2}\widetilde{\lambda }+f+1,\text{
}\widetilde{r}=\widetilde{k}=f\widetilde{\lambda }+1.
\end{equation*}

Illustrating a general construction from association schemes, I. M.
Chakravarti and W. C. Blackwelder \cite[Section 5]{CB} presented a family of
quasi-derived designs with parameters $%
(m^{2},3m^{2},m^{2}-1,(m^{2}-1)/3,(m^{2}-4)/3)$. These lie on the $F_{1}$%
-line 
\begin{equation*}
F_{1}(3,1):v=3\lambda +4,\text{ }b=9\lambda +12,\text{ }r=3\lambda +3,\text{ 
}k=\lambda +1.
\end{equation*}%
In fact, \emph{all} of the designs on this line in the range $2\leq \lambda
\leq 12$ covered by CRC exist (the one for $\lambda =1$ is the complete
design of 2-subsets of a 7-set). Their source symmetric designs would lie on
the $F_{1}$-line%
\begin{equation*}
F_{1}(1,1/3):\widetilde{v}=\widetilde{b}=9\widetilde{\lambda }+4,\text{ }%
\widetilde{r}=\widetilde{k}=3\widetilde{\lambda }+1.
\end{equation*}%
But not all those sources exist. For example, if $\widetilde{v}=9\lambda +13$
is even (so that $\lambda $ is odd), Sch\"{u}tzenberger's theorem demands
that $\widetilde{r}-\widetilde{\lambda }=2\lambda +3$ be a square. That
rules out $\lambda =5,7$, and $9$--but not $11$, the existence of a genuine $%
(112,37,12)$ design being listed as unknown. For the original designs in 
\cite{CB}, the corresponding symmetric designs have parameters $%
(3m^{2}+1,m^{2},(m^{2}-1)/3)$. When $m$ is even, the Diophantine equation of
the Bruck-Ryser-Chowla theorem is%
\begin{equation*}
Z^{2}=\frac{2m^{2}+1}{3}X^{2}+\frac{m^{2}-1}{3}Y^{2},
\end{equation*}%
and that has a solution $X=Y=1,Z=m$; so the needed symmetric design is not
ruled out. But when $m$ is odd, the order $(2m^{2}+1)/3$ must be a square.
That is, there must be an integer $l$ making $3l^{2}-2m^{2}=1$. The
solutions to this equation can be found by means of solutions to a Pell
equation, starting from the solution $l=m=1$. The first few values of $m$ are%
\begin{equation*}
\begin{tabular}{c|ccccccccc}
\multicolumn{10}{c}{$\text{TABLE VI. Solutions to a Pell-like equation}$} \\ 
\hline\hline
$j=$ & $0$ & $1$ & $2$ & $3$ & $4$ & $5$ & $6$ & $7$ & $8$ \\ 
$m_{j}=$ & $1$ & $11$ & $109$ & $1079$ & $10681$ & $105731$ & $1046629$ & $%
10360559$ & $102558961$%
\end{tabular}%
\end{equation*}%
In general, $m_{j}\approx (2+\sqrt{6})/4\times (5+2\sqrt{6})^{j}$.

\subsection{Difference families\label{subdifffam}}

In his comprehensive paper on cyclotomy and designs \cite{Wi}, R. M. Wilson
constructed difference families and corresponding designs in $\mathbb{F}_{q}$
under conditions that lead naturally to design lines. Theorem 7 of that
paper is this:

\begin{theorem}
\label{wilson}Let $k$ and $\lambda $ be given positive integers. If $%
2\lambda $ is a multiple of either $k$ or $k-1$, then for prime powers $%
q\geq k$, 
\begin{equation*}
\lambda (q-1)\equiv 0\,(\bmod\ k(k-1))
\end{equation*}%
is a necessary and sufficient condition for the existence of a $(q,k,\lambda
)$-difference family in $\mathbb{F}_{q}$.
\end{theorem}

\noindent The development of such a difference family, which is the set of
translates of the family members by elements of $\mathbb{F}_{q}$, is a
design with the given $k$ and $\lambda $ and with $v=q$.

To see a connection with design lines, let $k$ be the variable on the $F_{0}$%
-line%
\begin{equation*}
F_{0}(f,p):v=\frac{f(k-1)}{p}+1,\;b=\frac{f^{2}(k-1)}{p}+f=fv,\;r=fk,\;%
\lambda =pk.
\end{equation*}%
We seek designs on this line that satisfy the conditions of Theorem \ref%
{wilson}. As $2\lambda =(2p)k$, $2p$ must be an integer. Then since $\lambda
(v-1)=fk(k-1)$, $f$ must also be an integer. Finally, $v$ itself is to be a
prime power. For example, $F_{0}(2,1/2)$ is the line%
\begin{equation*}
F_{0}(2,1/2):v=4k-3,\;b=8k-6,\;r=2k,\;\lambda =k/2.
\end{equation*}%
All the designs on this line covered by CRC exist, $v$ a prime power or not,
except possibly $(69,138,36,18,9)$ and $(77,154,40,20,10)$. The first one
with $v$ a prime power other than a prime is $(125,250,64,32,16)$.

Similarly, the line $F_{1}(f,p)$ with $k$ as variable is 
\begin{equation*}
F_{1}(f,p):v=\frac{fk}{p}+1,\;b=\frac{f^{2}k}{p}+f,\;r=fk,\;\lambda =p(k-1).
\end{equation*}%
Thus $\lambda (v-1)=fk(k-1)$ and $2\lambda =2p(k-1)$. Once again for Theorem %
\ref{wilson}, we need $f$ and $2p$ to be integers. The analogous $F_{1}$%
-line is%
\begin{equation*}
F_{1}(2,1/2):v=4k+1,\;b=8k+2,\;r=2k,\;\lambda =\frac{k-1}{2}.
\end{equation*}%
This time, all that are listed in CRC exist, except perhaps $%
(77,154,38,19,9) $.

Earlier, D. A. Sprott presented several series of designs also compatible
with design lines \cite{Sp1,Sp2}, basing his results on the fundamental work
of R. C. Bose \cite{Bo} (see also \cite[Section VI.16]{CD}). Wilson's later
theorem covers some of these designs. The three types, $P$, $F_{0}$, and $%
F_{1}$, all appear in Sprott's papers. For instance, the designs of series 1
in \cite{Sp2} lie on $F_{0}$-lines%
\begin{equation*}
F_{0}(f,1):v=f(k-1)+1,\;b=f^{2}(k-1)+f,\;r=fk,\;\lambda =k,
\end{equation*}%
with $v=q$. As an example when $f=3$, CRC lists designs with $v=q$ for all
the prime powers $q$ for which $q\equiv 1\,(\bmod\ 3)$ and for which the
design is in the table range As with the previous examples, it also lists
designs with $v$ not prime powers, namely $v=10,22,28$, and $34$.
Incidentally, notice that if $\lambda |k$, then $k/\lambda \times
(v,b,r,k,\lambda )=(v,bk/\lambda ,rk/\lambda ,k,k)$. Many of the designs
listed in CRC with $k=\lambda $ have realizations as replicated designs.

\subsection{Affine designs\label{subaff}}

An \textbf{affine} design is one having a parallelism for which blocks in
different parallel classes always meet in the same number of points (see 
\cite[p. 16]{CvL}, for example). By a theorem of R. C. Bose \cite{B}, a
design with a parallelism is affine exactly when $r=k+\lambda $, the
quasi-residual condition. The first design on $F_{0}(3/2,1/2)$ (\ref{f2}) in
Subsection \ref{subqr} above, $(4,6,3,2,1)$, is the affine plane of order 2.
But a design with a parallelism must evidently have $k|v$ also, and this
affine plane is the only design on $F_{0}(3/2,1/2)$ line that satisfies this
additional restriction. On the other hand, the $P$-line on $(4,6,3,2,1)$ is%
\begin{equation}
P(2,1):v=2k,\;b=4k-2,\;r=2k-1,\;\lambda =k-1.  \label{p2}
\end{equation}%
All of these designs satisfy $r=k+\lambda $, of course, as well as the
divisibility condition $k|v$. However, an affine design must further have $%
v|k^{2}$. On this line, this condition is that $k$ be even. From CRC, each
such design with such $k\leq 20$ does indeed have an affine model.

The line $P(2,1)$ is the line for design series 4 in \cite{Sp2}. The designs
constructed there, for which $v=q+1$ for a prime power $q\equiv 3\,(\bmod\
4) $, are also 3-designs (so are those of series 5): each triple of points
is contained in $\lambda _{3}=(\lambda -1)/2$ blocks. In fact, any design on
this line having $4|v$ has a realization as an Hadamard 3-design \cite[%
Construction V.1.1.7]{CD}, if the needed Hadamard matrix of order $v$
exists. Each of these designs is \emph{formally} self-complementary: $\Delta
C=\Delta $ ($\Delta C=\Delta $ is not to say that if $\Delta $ represents a
genuine design it has a self-complementary realization).

\subsection{3-designs}

This raises an issue with 3-designs more generally: are there design lines
on which $\lambda _{3}$ is also linear? With $\lambda _{3}=\lambda
(k-2)/(v-2)$, the conditions for linearity on the three line types are%
\begin{equation*}
P:p=f/2,\;F_{0}:p=-f,\;F_{1}:p=2f
\end{equation*}%
($f=p$ is an algebraic possibility, but it is excluded from these line
types). So the only lines with the added linearity that could contain
genuine 3-designs are the $P(2p,p)$ (on $F_{1}(f,2f)$, $v=k/2-1$, and on $%
F_{0}(f,-f)$, $v=2-k$). In (\ref{p2}) we saw $P(2,1)$ for Hadamard
3-designs. The $P$-line%
\begin{equation*}
P(4,2):v=\lambda +2,\;b=4\lambda +4,\;r=2\lambda +2,\;k=\frac{\lambda }{2}+1,
\end{equation*}%
contains the designs of series 5 in \cite{Sp2}, 3-designs for which $v-1$ is
a prime power.

However, suppose that $\mathcal{H}$ is a genuine Hadamard 3-design on $%
P(2,1) $, with parameters $(2\lambda +2,4\lambda +2,2\lambda +1,\lambda
+1,\lambda ) $. As $\lambda _{3}=(\lambda -1)/2$, $\lambda $ must be odd.
The complement of any block of $\mathcal{H}$ is also a block, and if two
different blocks meet nontrivially, their intersection has size $(\lambda
+1)/2$; both facts are recorded in the proof of \cite[Proposition 6.3.1]{IS}%
. As for the derived design of a symmetric design, construct the design $%
\mathcal{H}^{\prime }$ with a block $\mathbf{B}$ of $\mathcal{H}$ for point
set, and let its block set comprise those block intersections $\mathbf{B}%
\cap \mathbf{B}^{\prime }$ for which $\left\vert \mathbf{B}\cap \mathbf{B}%
^{\prime }\right\vert =(\lambda +1)/2$. The parameters of $\mathcal{H}%
^{\prime }$ are $(\lambda ^{\prime }+2,4\lambda ^{\prime }+4,2\lambda
^{\prime }+2,\lambda ^{\prime }/2+1,\lambda ^{\prime })$, where $\lambda
^{\prime }=\lambda -1$, and $\mathcal{H}^{\prime }$ also lies on $P(4,2)$.
Moreover, $\mathcal{H}^{\prime }$ inherits the 3-design property of $%
\mathcal{H}$, with $\lambda _{3}^{\prime }=\lambda ^{\prime }/2-1$. These
3-design parameters are listed in \cite[Table II.4.37]{CD} under the extra
requirement that $4|\lambda ^{\prime }$, which guarantees there are no
replicated blocks (as one infers from the original Hadamard matrix). Designs
with these same parameters--but not specified as 3-designs--constructed from
conference matrices are presented in \cite[Theorem 5.4]{LS}; but see also 
\cite{GHH}, where the 3-design properties are made explicit. Incidentally, $%
P(2,1)$ $M_{2}=P(4,2)$, and if $\Delta $ is a genuine 3-design on $P(2,1)$, $%
2\times \Delta $ will be a genuine 3-design on $P(4,2)$ (albeit with
repeated blocks).

\subsection{Family (A)}

In \cite{Si}, G. P. Sillitto showed that designs with $b=4(r-\lambda )$ have
a multiplicative property: two such designs with $v=v_{1}$ and $v=v_{2}$ can
be combined to give a design of the same sort with $v=v_{1}v_{2}$. (If $%
(1,-1)$ incidence matrices are used, the matrix for the product is the
Kronecker product of the matrices of the factors.) Labeling these designs as
forming \textbf{family (A)}, S. S. Shrikhande investigated them further \cite%
{Sh} (see also \cite[Section 15.7]{H}). The replicate line on a design $%
\Delta _{0}$ in the family clearly comprises designs in family (A), but
direct substitution shows that $\Delta _{0}$ itself is the only design in
family (A) on its other $\mathcal{D}$-lines.

Each $\mathcal{D}$-line contains exactly one design from family (A) except
for the lines with $f=2p$. The designs on $P(2p,p)$ satisfy $v=2k$, and $%
P(2p,p)C=P(2p,p)$. Those on $F_{0}(2p,p)$ satisfy $v=2k-1$, and their
complements on $F_{1}(2p,p)$ have $v=2k+1$. All of these designs are closely
related to Hadamard matrices (see \cite{K}, for example).

\section{Existence\label{sectexist}}

What can be said about the existence of genuine designs along $\mathcal{D}$%
-lines? Some possibilities have been implied in Subsection \ref{subdifffam}.
A general asymptotic result is probably not possible. For example, for the
line of Hadamard designs, asymptotic existence is tantamount to existence of
Hadamard matrices. The $F_{0}$-line on $(7,3,1)$, $F_{0}(1,1/3)$, has $%
v=b=(9Q-4)/2,\;r=k=3Q/2,\;\lambda =Q/2$. The $v$ even case of the
Bruck-Ryser-Chowla theorem applies here when $4|Q$. As $n=Q$, $Q$ must be a
square. Thus arbitrarily far along this $F_{0}$-line there are designs with
integral components that do not have models.

On the other hand, there is a construction of generalized designs due to
Graver and Jurkat \cite{GJ} and Wilson \cite{W} (who uses the term \emph{%
pseudo-design}): let $X$ be the point set for a design, with $\left\vert
X\right\vert =v$. To each $k$-subset $Y$ of $X$ associate an integer $c_{Y}$%
, representing the multiplicity of $Y$ as a block of the design. When all
the $c_{Y}$ are nonnegative, the structure to be defined will be a standard
design, except that if $c_{Y}>1$, then $Y$ is a repeated block appearing $%
c_{Y}$ times. The requirement on the $c_{Y}$ is that for $x,y\in X$, $x\neq
y $, $\sum_{x,y\in Y}c_{Y}$ is a constant $\lambda $ independent of $x$ and $%
y$. If $k>1$, it follows that for $x\in X$, $\sum_{x\in Y}c_{Y}$ is also a
constant $r$ independent of $x$ (when $k=1$ we impose this constancy in the
set-up). With $b=\sum c_{Y}$, one then has $(v,b,r,k,\lambda )\in \mathcal{D}
$. The cited papers show that given $v$, $k$, and $\lambda $, with $v\geq
k\geq 2$, then if $(k-1)|\lambda (v-1)$ and $k(k-1)|\lambda v(v-1)$, there
is a multiplicity function $Y\rightarrow c_{Y}\in \mathbb{Z}$ satisfying the
requirement. Thus with just the inequalities $v\geq k\geq 2$, an integer
point $(v,b,r,k,\lambda )$ on $\mathcal{D}$ at least represents a
pseudo-design.

By way of an example, we saw in Subsection \ref{subqr} that there is no
genuine $(15,21,7,5,2)$ design. On the other hand, CRC reveals that the
designs $\Delta _{1}=(15,42,14,5,4)$ and $\Delta _{2}=(15,63,21,5,6)$ do
exist (in abundance). Set up realizations of $\Delta _{1}$ and $\Delta _{2}$
on the same 15-set $X$, and define the multiplicity function $c$ by $%
c_{Y}=-1 $ for $Y$ a block of $\Delta _{1}$, $c_{Y}=1$ for $Y$ a block of $%
\Delta _{2} $, and $c_{Y}=0$ for all other 5-subsets of $X$. Then the result
is a pseudo-design with parameters $(15,21,7,5,2)$.

\end{document}